\documentclass[reqno,makeidx,12pt]{amsart}

\usepackage{graphicx}
\graphicspath{{figs/}}
\usepackage{subfigure}
\usepackage{latexsym}
\usepackage{amstext}
\usepackage {amsmath}
\usepackage {amsfonts}
\usepackage {amssymb}
\usepackage {amsthm}
\usepackage {bbm}
\usepackage{enumerate}
\usepackage{array}
\usepackage{comment}
\usepackage{xspace}
\usepackage[bookmarks=true]{hyperref}
\usepackage{enumerate}
\ifx\cs\documentclass \footheight 12pt \fi \footskip 30pt 
\DeclareMathOperator{\dist}{dist}
\DeclareMathOperator{\sgn}{sgn}
\DeclareMathOperator{\spt}{supp}

\DeclareMathOperator\Id{Id}

\newcommand{\eps}{\varepsilon}

\newcommand{\Chi}{\mathcal{X}}
\newcommand{\R}{\ensuremath{\mathbb{R}}}

\newcommand{\N}{\ensuremath{\mathbb{N}}}

\newcommand{\Ha}{\ensuremath{\mathcal{H}}}

\def\R{\mathbb R}

\def\Bc{\mathcal{B}}
\def\Beps{\mathcal{B}_\eps}

\def\F{\mathcal{F}}

\theoremstyle{plain}
\numberwithin{equation}{section}
\newtheorem{lemma}{Lemma}[section]

\newtheorem{proposition}[lemma]{Proposition}

\theoremstyle{definition}

\begin{document}
\title[]{Confined elastic curves}
\author{Patrick W.~Dondl}
\address{Patrick W.~Dondl, Institut f\"ur Angewandte Mathematik,
Universit\"at Bonn,
Endenicher Allee 60,
D-53119 Bonn}
\email{pwd@hcm.uni-bonn.de}

\author{Luca Mugnai}
\address{Luca Mugnai, Max Planck Institute for Mathematics in the
  Sciences, Inselstr. 22, D-04103 Leipzig}
\email{mugnai@mis.mpg.de}

\author{Matthias R{\"o}ger}
\address{Matthias R{\"o}ger, Technische Universit\"at Dortmund,
Fakult\"at fŸr Mathematik,
Vogelpothsweg 87,
D-44227 Dortmund}
\email{matthias.roeger@tu-dortmund.de}



\date{\today}

\begin{abstract}
We consider the problem of minimizing Euler's elastica energy for simple closed curves confined to the unit disk. We approximate a simple closed curve by the zero level set of a function with values $+1$ on the inside and $-1$ on the outside of the curve. The outer container now
becomes just the domain of the phase field. Diffuse approximations of the elastica energy and the curve length are well known. Implementing the topological constraint thus becomes the main difficulty here. We propose a solution based on a diffuse approximation of the winding number, present a proof that one can approximate a given sharp interface using a sequence of phase fields, and show some numerical results using finite elements based on subdivision surfaces.
\end{abstract}

\maketitle
\section{Introduction}
\label{sec:intro}
Elastic structures confined to a certain volume or area appear in many situations. For example inner membranes in biological cells separate an inner region from the rest of the cell and consist of an elastic bilayer. The inner structures are confined by the outer cell membrane. Since the inner membrane contributes to the biological function it is advantageous to include a large membrane area in the cell. In two dimensions elastic structures confined to a plane ball have been experimentally produced by Bou\'{e} \emph{et alii} \cite{BABC06} (see also \cite{DGS02, SWH08}). They show that with increasing length the structures become more and more complex. 
We are considering here the problem corresponding to a one-dimensional closed elastic wire constrained in a  two-dimensional container of circular shape. More precisely we consider a wire  whose equilibrium (i.e. stress-, strain-free) configuration is given by a circle of radius $L/2\pi$, and we suppose that  both the friction between the wire and the boundary of the container,  and the friction between  portions of the wire that are in contact are negligible. We are interested in finding (stable) shapes  of the folded wire constrained in the container. More precisely we are interested in those shapes that are obtainable via pure bending deformation processes starting from the equilibrium configuration (in particular no stretching is allowed). 

We adopt the following mathematical description of the problem. We represent the container by the closed unit disk $\overline{B_1(0)}:=\{z\in\R^2:~\vert z\vert\leq 1\}$ and the folded (isotropic) elastic wire by an immersion $\gamma:S^1_L\,\to\,\overline{B_1(0)}$ of the circle $S^1_L$ of radius $L/2\pi$ in the unit closed disk. As for the bending elastic energy, we consider  the classical Euler's elastica energy  associated to the immersion $\gamma$. The configurations we are interested in  correspond to the (local) minimizers of the bending energy among the closed curves that are supported in the unit ball, and that can be reached via a deformation path that starts from the circle $S^1_L$, and along which the following three constraints are fulfilled: the length of the immersed curve remains equal to $L$ (so that we exclude stretching of the wire);   the elastic (bending) energy is uniformly bounded; the immersed curve may have multiple self-contact points, but does not have ``self-crossings'' (as this would correspond to  self-interpentration of the wire). It turns out (see Section \ref{sec:sharp-interf}) that the class of immersed curves satisfying the above constrains corresponds to the closure (with respect to the $W^{2,2}$-weak topology) of length-preserving diffeomorphisms of $S^1_L$ into $B_1(0)$. 
 In this formulation there are several intrinsic difficulties. Minimizers (for large prescribed length) are expected to have multiple touching points. Therefore the associated Euler-Lagrange equation  involves several Lagrange multipliers and an explicit characterization of the class of curves in which the minimum  is attained is difficult to  obtain. Furthermore the constraints of being confined to the unit ball and of not developing ``self-crossings'' are difficult to maintain in a steepest descent method.

In this paper we propose a \emph{phase field approximation} of the above problem. We justify our approach by an asymptotic analysis and investigate the problem by numerical simulations. As we already remarked above, in the original \emph{sharp interface} formulation admissible configurations correspond to immersions that can be approximated by a sequence of simple and closed curves.  Since simple and closed curves  bound an inner set we can approximate such sets by smooth fields with values close to $+1$ inside and close to $-1$ outside. Prescribing the confinement condition is now  rather simple: the outer container just becomes the domain of definition for the approximating phase fields and a boundary condition ensures that emerging structures do not leave the domain. An approximation of Euler's elastica energy is well known, implementing the topological condition thus becomes the main difficulty. One neccessary condition is that the phase field approximation of the winding number has to be close to $2\pi$. We will use a gradient flow for a relaxed diffuse approximation of the elastica energy that includes \emph{soft constraints} for the prescribed length and for the winding number. For ``generic situations'' we observe that this is sufficient to keep the right topology,  and avoid that phase interfaces cross transversally. However, in general this method does not exclude that a phase disconnects into several pieces.  To deal with this issue we show that an additional variable can detect multiple components and can be used to prevent structures from disconnecting.

Let us remark that the same phase-fields approximation  we use in this paper for the bending energy, has been successfully used in similar contexts (\textit{e.g.} \cite{DuWang04, DuWang05,  DuWang07, CampHern, CampHern2}). The main differences between our results and the previous literature are on the one hand the numerical methods we develop to solve the (diffuse interface) gradient flow, on the other hand the inclusion of the topological constraint in the energy.

The plan of the paper is the following. In Section \ref{sec:sharp-interf} we discuss the constrained minimization problem in its sharp interface formulation. In Section \ref{sec:di} we introduce the diffuse interface approximation. In Section \ref{sec:limsup} we will prove that we can approximate a given sharp interface configuration with a sequence of phase fields.  In Section \ref{sec:numerics} numerical simulations are presented that show that our approach works reasonably well for ``well-behaved'' initial data. A more exotic example shows that our topological constraint is in general not sufficient to enforce the correct topology for phase boundaries. In the Appendix we therefore propose an improved formulation for the constrained problem and indicate why this will lead to the correct result. 
 
%
\section{The sharp interface minimization problem}
\label{sec:sharp-interf}
We first discuss the minimization problem in its sharp interface formulation. Consider the unit ball $B_1(0)\subset\R^2$, a given length constraint $L>0$, and define the following class of admissable curves
\begin{gather}\label{eq:def-ML}
  M_L\,:=\, \Big\{ \gamma:[0,L]\to B_1(0),\, \gamma \text{ is a  closed and
    simple  }C^2\text{-curve},\,|\gamma'|=1\Big\}.
\end{gather}
In particular, elements of $M_L$ can be represented by $C^2$-diffeomorphisms of the standard sphere.
For $\gamma\in M_L$ Euler's elastica energy is given by 
\begin{gather}
  \Bc(\gamma)\,:=\, \int_0^L |\gamma''(s)|^2\,ds. \label{eq:def-B}
\end{gather}
We then consider the constrained minimization
problem: find the optimal value
\begin{gather}
  m_L\,:=\, \inf_{\gamma\in M_L} \Bc(\gamma), \label{eq:def-mL}
\end{gather}
and characterize minimal sequences and possible limit points. Since we expect touching points for the optimal structures, minimizers will in general not belong to the class $M_L$. However, we do obtain the following compactness property.
\begin{proposition}\label{prop:si-cpct}
Let $(\gamma_k)_{k\in\N}$ be a minimal sequence in $M_L$. Then there exists $\gamma\in W^{2,2}([0,L])$, such that
\begin{gather}
  \gamma_k\,\to\,\gamma \text{ weakly in }W^{2,2}([0,L]) \label{eq:conv-gamma-weak}
 \end{gather}
for a subsequence $k\to\infty$. The curve $\gamma$ has the following properties: 
 \begin{gather}
    \gamma \text{ is } C^1-\text{closed,} \label{eq:pr1}\\
  \gamma([0,L]) \,\subset\, \overline{B_1(0)}, \label{eq:pr2}\\
  |\gamma'(s)|\,=\, 1 \text{ for all }s\in[0,L], \label{eq:pr3}\\
  \gamma \text{ can touch the unit circle only tangentially}, \label{eq:pr4}\\
  \gamma \text{ has no transversal crossings}, \label{eq:pr5}
\end{gather}
where the last property means that $\gamma$ can touch itself only tangentially. Furthermore,  if we denote by $\Chi_k:\R^2\,\to\, \{0,1\}$ the characteristic function of the open subset of $B_1(0)$ that is enclosed by $\gamma_k$ we obtain that
\begin{gather}
  \Chi_k\,\to\, \Chi \text{ strongly in }L^1(\R^2). \label{eq:conv-Chi}
\end{gather}
The limit characteristic function $\Chi$ has the following properties:
\begin{gather}
  \Chi \,=\, \Chi_E, \text{ where }E\subset \overline{B_1(0)}\text{ is a set of finite perimeter}, \label{eq:pr6}\\
  \partial^*E \,\subset\, \spt(\gamma). \label{eq:pr7}
\end{gather}
Finally, $\gamma$ lies always on the same side of $E$: after changing the orientation of $\gamma$ if neccessary,
\begin{gather}
  \nu_E(x)\,=\, \sum_{\gamma(s)=x} \gamma'(s)^\perp \quad\text{ for }\Ha^1\text{-almost all }x\in \partial^*E, \label{eq:pr8}
\end{gather}
where in the last equation $\perp$ denotes the clockwise rotation by $\pi/2$ and $\nu_E$ the inner unit normal of $E$.
\end{proposition}
\begin{proof}
By the minimizing property of $\gamma_k$ we have that there exists $\Lambda>0$ such that
\begin{gather*}
 \Bc(\gamma_k) \,\leq\, \Lambda.
\end{gather*}
We moreover assume that all $\gamma_k$ are parametrized by arclength. Since $\gamma$ maps to the unit ball we therefore have a uniform bound of the sequence $(\gamma_k)_{k\in\N}$ in $W^{2,2}(0,L)$. We therefore deduce \eqref{eq:conv-gamma-weak} and by
Sobolev embedding Theorem that $\gamma\in C^{1,1/2}([0,L])$ and that
\begin{align}
  \gamma_{k}\,&\to\, \gamma \quad\text{ strongly in }C^{1,\alpha}([0,L]) \text{ for all }0\leq\alpha<\frac{1}{2}. \label{eq:conv-gamma-strong}
\end{align}
This also implies that \eqref{eq:pr1}-\eqref{eq:pr5} holds.

For the inner sets we have a uniform area bound by the confinement constraint and a uniform bound on the perimeter, that has length $L$ by the length constraint. Therefore $\Chi_k$ is uniformly bounded in $BV(\R^2)$ und we deduce that for a subsequence \eqref{eq:conv-Chi} holds and that $\Chi$ is the characteristic set $E$ satisfying \eqref{eq:pr6},\eqref{eq:pr7}. 

Finally we can orient all $\gamma_k$ such that $\gamma'(s)^\perp$ equals the inner normal of the set that is enclosed by $\gamma_k$.  Then we obtain for any function $\eta\in C^1_c(\R^2)$ from the Gau{\ss} Theorem and by \eqref{eq:conv-gamma-strong} that
\begin{align*}
  -\int_{\R^2} \nu_E(x)\cdot \eta(x)\,d|\nabla\Chi|(x) \,&=\, \int_{\R^2} \Chi(x)\nabla\cdot\eta(x)\,dx\\
  &=\, \lim_{k\to\infty}  \int_{\R^2} \Chi_k(x)\nabla\cdot\eta(x)\,dx\\
  &=\,-\lim_{k\to\infty} \int_0^L \gamma_k'(s)^\perp\cdot\eta(\gamma_k(s))\,ds\\
  &=\, \int_0^L \gamma'(s)^\perp\cdot\eta(\gamma_k(s))\,ds,
\end{align*}
which proves \eqref{eq:pr8} since $\eta$ was arbitrary.
\end{proof}
Proposition \ref{prop:si-cpct} in particular shows that minimizers of \eqref{eq:def-mL} belong to the closure $\overline{M_L}$ of $M_L$ with respect to the weak-$W^{2,2}([0,L])$-topology.  For our purposes, however, we only need the following alternative characterization of $\overline{M_L}$: curves in $\overline{M_L}$ can be approximated \emph{strongly} in $W^{2,2}$ by closed simple curves that are strictly contained in the unit ball.
\begin{proposition}\label{prop:alt-char}
For any $\gamma\in \overline{M_L}$ there exists a sequence $(\gamma_k)_{k \in\N}$ of simple closed $C^2$-curves with 
\begin{gather}
  \gamma_k\,\to\, \gamma\quad\text{ as }k\to\infty\text{ strongly in }W^{2,2}([0,L]), \label{eq:appro1}\\
  |\gamma_k'|(s)\,=\, 1 \quad\text{ for all }s\in [0,L], \label{eq:appro2}\\
  \gamma_k([0,L])\,\subset\, B_1(0). \label{eq:appro3}
\end{gather}
\end{proposition}
\begin{proof}
(i) We first assume that $\gamma([0,L])\subset B_1(0)$ and therefore that
\begin{gather*}
  \delta\,:=\, \dist(\gamma, S_1(0))\,>\,0.
\end{gather*}
Repeating the proof of  \cite[Corollary 5.2]{BeMu04} under the additional hypothesis that, being $\gamma\in \overline{M_L}$,  $\gamma$ is the $W^{2,2}$-weak-limit of a sequence of diffeomorphisms of the unit circle with equi-bounded ``bending-energy'', we obtain the existence of  a sequence $(c_k)_{k\in\N}$ of simple closed $C^2$-curves and a sequence $(\lambda_k)_{k\in\N}$ of positive numbers such that,
\begin{align}
   c_k\,&\to\, \gamma\quad\text{ as }k\to\infty\text{ strongly in }W^{2,2}([0,L]), \label{eq:ck1}\\
  |c_k'|(s)\,&=\, \lambda_k \quad\text{ for all }s\in [0,L]. \label{eq:ck2}
\end{align}
From \eqref{eq:ck1} it follows that $\lambda_k\to 1$ as $k\to\infty$.
Let
\begin{gather*}
  \delta_k\,:=\, \dist(c_k, S_1(0))
\end{gather*}
then \eqref{eq:ck1} yields $\delta_k\to \delta$ as $k\to\infty$. We now define
\begin{gather*}
  \gamma_k(s) \,:=\, \frac{1}{\lambda_k}c_k(s),\quad s\in [0,L]
\end{gather*}
and observe that $\gamma_k$ is a simple closed $C^2$-curve with
$|\gamma_k'(s)=1|$ for all $s\in [0,L]$. Moreover we have
\begin{gather*}
  \gamma_k \,\to\, \gamma \quad\text{ strongly in } W^{2,2}(0,L),\\
  \dist(\gamma_k,S_1(0))\,=\, 1-\frac{1-\delta_k}{\lambda_k}\,\to\, \delta,
\end{gather*}
in particular $\gamma_k([0,L])\subset B_1(0)$ for $k$ large enough. Therefore $(\gamma_k)_{k\in\N}$ has all required properties.
\\[1ex]
(ii)
We next consider the general case $\gamma([0,L])\subset \overline{B_1(0)}$. First we observe that $\gamma([0,L])\cap B_1(0)$ cannot be empty for $L> 2\pi$ (for $L\leq2\pi$ any minimizing sequence converges to a parametrization of the circle with length $L$). In fact, assume the contrary and let $(c_l)_{l\in\N}$ be a sequence in $M_L$ approximating $\gamma$ weakly in $W^{2,2}([0,L])$. Then, using Gau\ss\ Theorem
\begin{gather*}
  L\,=\, \lim_{l\to\infty} \int_0^L \nu_l(s)\cdot\gamma_l(s)\,ds \,=\, \lim_{l\to\infty} \int_{\R^2} \Chi_l(x) \nabla\cdot x\,dx \,=\, 2\pi, 
\end{gather*}
which is a contradiction. Therefore $\gamma$ cannot be entirely contained in the unit circle. 

Next consider $\xi\in C^\infty_c(B_1(0))$. Then there exists $t_0>0$ and a smooth evolution of diffeomorphisms $\varphi_t: B_1(0)\,\to\, B_1(0)$, $t\in (-t_0,t_0)$, such that
\begin{gather*}
  \frac{\partial}{\partial t}\varphi_t(s)\,=\, \xi(\varphi_t(s))\text{ for all }s\in [0,L],\qquad \varphi_0\,=\, \Id.
\end{gather*}
Then $c_t(s)\,:=\, \varphi_t(\gamma(s))$ defines a smooth evolution of $C^1$-closed $W^{2,2}$ curves such that all $c_t$ are contained in $\overline{B_1(0)}$ and such that $c_t\,\to\, \gamma$ strongly in $W^{2,2}([0,L])$ as $t\to 0$. We compute for the length
  $L(t)\,:=\, \int_0^L |c_t'(s)|\,ds$
that
\begin{gather}
  \frac{d}{dt}\Big|_{t=0} L(t) \,=\, -\int_0^L \gamma''(s)\cdot\xi(\gamma(s))\,ds \label{eq:var-Lt}
\end{gather}
and observe that this expression cannot vanish for all $\xi\in C^\infty_c(B_1(0))$ since otherwise $\gamma([0,L])\cap B_1(0)$ consists of a collection of straight lines, which contradicts the fact  that $\gamma$ can touch $S_1(0)$ only tangentially. Therefore we find $\xi\in C^\infty(B_1(0))$ and $t_0>0$ such that the length of $c_t$ is strictly increasing on $[0,t_0)$ and such that $c_t\,\to\, \gamma$ strongly in $W^{2,2}([0,L])$ as $t\searrow 0$. 
In the following we fix such $\xi$ and $t_0$ and define modified curves with length $L$, 
\begin{gather*}
  \gamma_t(s)\,:=\, \frac{L}{L(t)} c_t(\sigma(s)),\quad s\in [0,L]
\end{gather*}
where $\sigma(s)$ denotes the arclength reparametrization, such that $|\gamma_t'|\equiv 1$ holds. Then $\gamma_t$ is strictly contained in $B_1(0)$. Moreover, we claim that $\gamma_t\in \overline{M_L}$. In fact let $(c_l)_{l\in\N}$ be a sequence in $M_L$ approximating $\gamma$ weakly in $W^{2,2}([0,L])$. Define curves $c_{l,t}(s):=\varphi_t(c_l(s))$ according to the variation field $\xi$ fixed above. Then it follows from \eqref{eq:var-Lt}, the choice of $\xi$, and the weak $W^{2,2}$ convergence of $c_l$ to $\gamma$ that
\begin{gather*}
  \frac{d}{dt}\Big|_{t=0} L(l,t) \,=\, -\int_0^L c_l''(s)\cdot\xi(c_l(s))\,ds\,>\, 0
\end{gather*}
for all $l$ large enough. We then set
\begin{gather*}
  \gamma_{l,t}(s)\,:=\, \frac{L}{L(l,t)} c_{l,t}(\sigma_l(s)),\quad s\in [0,L]
\end{gather*}
as above and obtain that $\gamma_{l,t}\in M_L$. Moreover we have that $\gamma_{l,t}\to \gamma_t$ as $l\to\infty$ weakly in $W^{2,2}([0,L])$, hence $\gamma_t\in \overline{M_L}$. Thus, we can apply part (i) and obtain a sequence of $\gamma_{t,k} \in M_L$ that approximates $\gamma_t$ strongly in $W^{2,2}$. Taking a diagonal sequence proves the claim in the general case.
\end{proof}
\section{The diffuse interface approximation}
\label{sec:di}
Phase field approximations of sharp interface problems are widely used for numerical simulations and arise from mean field descriptions of phase separation processes in various applications. In the following $u:B_1(0)\to\R$ is a smooth function. The basis of the phase field formulation is an interfacial energy of the form
\begin{gather}
  L_{\varepsilon}(u)\,:=\,\frac{1}{c_0}\int_{B_1(0)}\Big(\frac{\varepsilon}{2}|\nabla
  u|^2  
  +\frac{1}{\varepsilon}W(u)\Big) dx. \label{eq:def-Leps}
\end{gather}
Here $\eps>0$ is a small parameter and 
$W$ denotes the standard quartic double-well potential
\begin{gather}
  W(r)\,=\,\frac{1}{4}(1-r^2)^2. \label{eq:def-W}
\end{gather}
It is well-known \cite{MoMo77} that $\frac{1}{c_0}L_\eps$ approximates the curve length functional in the sense of Gamma-convergence, where
\begin{gather}\label{eq:def-c0}
  c_0 \,:=\, \, \int_{-1}^1 \sqrt{2W(s)}\,ds.
\end{gather}
A phase-field analogue of Euler's elastica energy was already proposed by De Giorgi \cite{DeG91}.  For the modified version
\begin{gather}
  \Beps(u) \,=\, \frac{1}{c_0} \int_{B_1(0)} \frac{1}{\eps}\Big(-\eps\Delta u +
  \frac{1}{\eps}W(u)\Big)^2 \label{eq:def-Beps}
\end{gather}
the approximation property was proved in two and three dimensions \cite{RSc06}.
Moreover, following \cite{BeMu10} we introduce the \emph{diffuse winding number}
\begin{gather}
  T_\eps(u) \,=\, \frac{1}{c_0} \int_{B_1(0)} \Big(-\eps\Delta u +
  \frac{1}{\eps}W(u)\Big)|\nabla u|. \label{eq:def-Teps}
\end{gather}
Finally we propose to approximate the constrained minimization problem \eqref{eq:def-mL} by the problem of minimizing
\begin{gather}
	\F_\eps(u)\,=\, \Beps(u) + \eps^{-\alpha}\Big(L_\eps(u)-L\Big)^2 + \eps^{-\beta}\Big(T_\eps(u)-2\pi\Big)^2 \label{eq:def-F}
\end{gather}
under the boundary conditions
\begin{gather}
  u(x)\,=\, -1, \quad \nabla u(x)\cdot  x\,=\, 0 \quad \text{ for all } |x|=1, \label{eq:bdry-cdt}
\end{gather}
that prevent diffuse interface from touching the outer container.

The existence of minimizers for \eqref{eq:def-F}, \eqref{eq:bdry-cdt} follows with the direct method of calculus of variations. Since we are interested in minimizers of the functional $\Bc$ the adequate statement regarding the relation between the sharp and diffuse minimization problems would be the Gamma-convergence of $\F_\eps$ to $\Bc$. Though we are not able to prove such result in full generality, nevertheless we do obtain a compactness result and a lower bound estimate in the case of a regular limit point as a consquence of \cite{RSc06} (see also \cite{ToYu07}).
\begin{proposition}\label{prop:si-limit}
Let $(u_\eps)_{\eps>0}$ be a sequence of smooth functions $u_\eps:B_1(0)\to\R$ that satisfy the boundary condition \eqref{eq:bdry-cdt} and assume that
\begin{gather}
  \sup_{\eps>0}\F_\eps(u_\eps)\,<\, \infty.
\end{gather}
Then there exists a set $E\subset {B_1(0)}$ of finite perimeter such that
\begin{gather}
  u_\eps\,\to\, 2\Chi_E -1 \quad\text{ strongly in }L^1(B_1(0)).
\end{gather}
Moreover, the diffuse interface measures
\begin{gather}
  \mu_\eps \,:=\, \frac{1}{c_0}\Big(\frac{\eps}{2}|\nabla u_\eps|^2 + \frac{1}{\eps}W(u_\eps)\Big)\,dx \label{eq:def-mu}
\end{gather}
converge in measure to a Radon measure $\mu$ with support in $\overline{B_1(0)}$. If $\mu$ is given by a curve $\gamma\in \overline{M_L}$ in the sense of
\begin{gather}
  \int_{\R^2} \eta(x)\,d\mu(x)\,=\, \int_0^L \eta(\gamma(s))\,ds\quad\text{ for all }\eta \in C^0_c(\R^2), \label{eq:mu-gamma}
\end{gather}
then 
\begin{gather}
  \Bc(\gamma)\,\leq\, \liminf_{\eps\to 0} \F_\eps(u_\eps)
\end{gather}
holds.
\end{proposition}
 In general, the limit measure $\mu$ will not be given by a curve in $\overline{M_L}$ but will enjoy some weak regularity (being an integral varifold with weak mean curvature in $L^2$). We will demonstrate in Appendix \ref{sec:impro} that $\mu$ can consist of several disjoint curves and therfore does not belong to $\overline{M_L}$. On the other hand we are mainly interested in the numerical simulation of a steepest descent evolution for $\F_\eps$ and this in fact works sufficiently well. In general a more complex functional is needed, and we propose in Appendix \ref{sec:impro} a possible choice.
\section{Construction of recovery sequences}
\label{sec:limsup}
Whereas we cannot prove that minimizer $u_\eps$ converge to curves $\gamma\in \overline{M_L}$ we can show that any such curve $\gamma$ can be approximated by a suitable \emph{recovery sequence}. This result also extends to the improved functional we propose in the next section and justifies our approximation of the sharp interface minimization problem. We first start with the most regular case.
\begin{lemma}\label{lem:limsup}
Let $\gamma\in M_L$ be given. Then there exists a
sequence $u_\eps:B_1(0)\to [-1,1]$ of smooth phase fields such that the diffuse interface measures $\mu_\eps$ (as defined in \eqref{eq:def-mu}) converge to the measure $\mu$ that is \begin{gather}
    \mu_\eps\,\to\, c_0 \Ha^1\lfloor \gamma \label{eq:limsup-conv-mu}
\end{gather}
as $\eps\to 0$. Furthermore for all $\eps>0$ holds
\begin{align}
  L_\eps(u_\eps)\,&=\, L + R_\eps^{(L)}, \label{eq:limsup-L}\\
  T_\eps(u_\eps)\,&=\,2\pi + R_\eps^{(T)}, \label{eq:limsup-T}\\
  \Beps(u_\eps)\,&=\, \Bc(\gamma) + R_\eps^{(B)}, \label{eq:limsup-B}
\end{align}
where $R_\eps^{(L)},R_\eps^{(T)}$ are exponentially small in $\eps>0$
and $R_\eps^{(B)}$ is of order $O(\eps^2)$. In particular,
\begin{gather}
  \F_\eps(u_\eps)\,\to\, \Bc(\gamma) \label{eq:limsup-conv}
\end{gather}
for any choice of $\alpha,\beta>0$ in \eqref{eq:def-F}.
\end{lemma}
\begin{proof}
The construction is standard and uses the optimal profile $q$ for the
one-dimensional minimisation in the Cahn--Hilliard energy, the signed
distance function from $\gamma$, and an interpolation to the
stationary points $\pm 1$. To be precise, let
$q:\R\to (-1,1)$ be the solution of
\begin{gather}
  -q^{\prime\prime} + W^\prime(q)\,=\, 0, \label{eq:phi}\\
  q(-\infty)\,=\, -1, \quad q(+\infty)\,=\,1, \quad q(0)\,=\,0.
  \label{eq:infty-phi_0}
\end{gather}
Then
\begin{gather}
  q'(r)\,=\,\sqrt{2W(q(r))} \label{eq:equipart-q}
\end{gather}
holds for all $r>0$ and with \eqref{eq:def-W} we have
\begin{gather}
  q(r) \,=\, \tanh \big(r/\sqrt{2}\big). \label{eq:q-tanh}
\end{gather}
Moreover
there exists $\delta>0$ such that signed distance function $d$ from
$\gamma$ 
(taken positive in the region inside of $\gamma$) is of class $C^2$.
Next fix  a smooth symmetric cut-off function $\eta\in C^\infty(\R)$,
\begin{gather*}
  0\leq \eta\leq 1,\quad \eta(r) = 1\text{ for }r\in[-1,1],\quad 
  \eta(r) = 0 \text{ for }|r|\geq 2,\quad \eta'\,\leq\,0.
\end{gather*}
We then define
\begin{gather*}
  q_\eps(r)\,:=\,
  \eta\big(\frac{2r}{\delta}\big)q\big(\frac{r}{\eps}\big) +
  \sgn(t)\big(1-\eta\big(\frac{2r}{\delta}\big) \big) 
\end{gather*}
and
\begin{gather}
  u_\eps(x)\,:=\,  q_\eps(d(x)). \label{eq:def-u-limsup}
\end{gather}
\underline{Step 1:} Consider the parametrization
\begin{gather}
  \psi: [0,L)\times (-\delta,\delta)\,\to\, B_1(0),\quad \psi(s,t)\,=\,
  \gamma(s)+t\nu(s),
\end{gather}
which is injective by the choice of $\delta$ and continuously
differentiable with
\begin{gather}
  \det D\psi(s,t)\,=\, 1- t\kappa(s).
\end{gather}
We then compute that
\begin{align}
  L_\eps(u_\eps)\,&=\, \int_0^L\int_{-\delta}^{\delta}
  \Big(\frac{\eps}{2} q_\eps'(t)^2 + \frac{1}{\eps}W(q_\eps(t) \Big)
  (1+t\kappa(s))\,dt\,ds \notag\\
  &=\, \int_0^L \int_{-\delta/2}^{\delta/2}  \frac{1}{\eps}\Big(\frac{1}{2}
  q'(t/\eps)^2 + W(q(t/\eps) \Big)(1+t\kappa(s))\,dt\,ds \notag\\
  &\quad + \int_0^L \int_{\{\delta/2<|t|<\delta\}}  \Big(\frac{\eps}{2}
  q_\eps'(t)^2 + 
  \frac{1}{\eps}W(q_\eps(t)) \Big) (1+t\kappa(s))\,dt \,ds. \label{eq:est-L0}
\end{align}
By the symmetry of $q$ and \eqref{eq:equipart-q}, \eqref{eq:q-tanh} we
obtain for the first 
integral on the right-hand side that
\begin{align}
  &\int_0^L\int_{-\delta/2}^{\delta/2}  \frac{1}{\eps}\Big(\frac{1}{2}
  q'(t/\eps)^2 + W(q(t/\eps) \Big)(1+t\kappa(s))\,dt \,ds\notag\\
  =\,
  & 2L\int_0^{\delta/{2\eps}} q'(t)\sqrt{2W(q(t))}\,dt \notag\\
  =\,& c_0 L - 2L\int_{q(\frac{\delta}{2\eps})}^1 \sqrt{2W(r)}\,dr \notag\\
  =\,& c_0L - \sqrt{2}L \Big( 1-\tanh\big(\frac{\delta}{2\sqrt{2}\eps}\big)\Big)
  -\frac{\sqrt{2L}}{3}\Big(1-\tanh^3\big(\frac{\delta}{2\sqrt{2}\eps}\big)
  \Big) \label{eq:est-L1}
\end{align}
Furthermore, using \eqref{eq:equipart-q} again
\begin{align*}
  q_\eps'\,&=\,
  \frac{2}{\delta}\eta'\big(\frac{2t}{\delta}\big)\Big(q\big(\frac{t}{\eps}\big)-1\Big)
  + \frac{1}{\eps}\eta\big(\frac{2t}{\delta}\big)q'\big(\frac{t}{\eps}\big)\\
  &=\,
  \Big(1-q\big(\frac{t}{\eps}\big)\Big)\Big(-\frac{2}{\delta}\eta'\big(\frac{2t}{\delta}\big) 
  +
  \frac{1}{\sqrt{2}\eps}\eta\big(\frac{2t}{\delta}\big)\Big(1+q\big(\frac{t}{\eps}\big)\Big)\Big). 
\end{align*} 
With this equality and the symmetry of $q_\eps$ we calculate for the
second integral in \eqref{eq:est-L0} that
\begin{align}
  &\int_{\{\delta/2<|t|<\delta\}}  \Big(\frac{\eps}{2}
  q_\eps'(t)^2 + 
  \frac{1}{\eps}W(q_\eps(t)) \Big) (1+t\kappa(s))\,dt \notag\\
  =\, &2\int_{\frac{\delta}{2}}^\delta
  \Big(1-q\big(\frac{t}{\eps}\big)\Big)^2
  \Big( \frac{\eps}{2}\Big(-\frac{2}{\delta}\eta'\big(\frac{2t}{\delta}\big) 
  +
  \frac{1}{\sqrt{2}\eps}\eta\big(\frac{2t}{\delta}\big)\Big(1+q\big(\frac{t}{\eps}\big)\Big)\Big)^2
  + \frac{1}{4\eps}\Big(1+q\big(\frac{t}{\eps}\big)\Big)^2\Big)\,dt \label{eq:est-L2}
\end{align}
Together with \eqref{eq:est-L1} we obtain \eqref{eq:limsup-L} with
\begin{align*}
  R_\eps^{(L)} \,=\, & -\sqrt{2}L \Big( 1-\tanh\big(\frac{\delta}{2\sqrt{2}\eps}\big)\Big)
  -\frac{\sqrt{2L}}{3}\Big(1-\tanh^3\big(\frac{\delta}{2\sqrt{2}\eps}\big)
  \Big)\\
  &+  2L \int_{\frac{\delta}{2}}^\delta
  \Big(1-q\big(\frac{t}{\eps}\big)\Big)^2
  \frac{\eps}{2}\Big(-\frac{2}{\delta}\eta'\big(\frac{2t}{\delta}\big) 
  +
  \frac{1}{\sqrt{2}\eps}\eta\big(\frac{2t}{\delta}\big)\Big(1+q\big(\frac{t}{\eps}\big)\Big)\Big)^2
  \,dt \notag\\
  & +2L \int_{\frac{\delta}{2}}^\delta
  \Big(1-q\big(\frac{t}{\eps}\big)\Big)^2\frac{1}{4\eps}\Big(1+q\big(\frac{t}{\eps}\big)\Big)^2\Big)\,dt, 
\end{align*}
which is exponentially small in $\eps>0$.
\\[2mm]
\underline{Step 2:} 
Let $\xi\in C^0(B_1(0))$. We compute that
\begin{align}
  \mu_\eps(\xi)\,&=\, \int_0^L\int_{-\delta}^{\delta}
  \Big(\frac{\eps}{2} q_\eps'(t)^2 + \frac{1}{\eps}W(q_\eps(t) \Big)
  \xi(\gamma(s)+t\nu(s))(1+t\kappa(s))\,dt\,ds \notag\\
  &=\, \int_0^L\int_{-\frac{\delta}{2\eps}}^{\frac{\delta}{2\eps}}
  \Big(q'(t)^2 + W(q)\Big)\xi(\gamma(s)+\eps t\nu(s))(1+\eps t\kappa(s))\,dt\,ds
  \notag\\
  &\qquad +\int_0^L\int_{\{\frac{\delta}{2}<|t|<\delta\}}
  \Big(\frac{\eps}{2} q_\eps'(t)^2 + \frac{1}{\eps}W(q_\eps(t) \Big)
  \xi(\gamma(s)+t\nu(s))(1+t\kappa(s))\,dt \,ds. \label{eq:est-mu1}
\end{align}
As above we conclude that the second term is expentially small in $\eps>0$ and that
\eqref{eq:equipart-q} we derive
\begin{gather*}
  \lim_{\eps\to 0} \mu_\eps(\xi)\,=\,c_0 \int_0^L
  \xi(\gamma(s))\,ds\,=\, c_0 \int_\gamma \xi \,d\Ha^{n-1},
\end{gather*}
which proves \eqref{eq:limsup-conv-mu}.
\\[2mm]
\underline{Step 3:}
From \eqref{eq:equipart-q} we obtain, using the shortcuts
$\eta=\eta(\delta^{-1}2t)$, $q=q(\eps^{-1}t)$ etc., that
\begin{align}
  &-\eps q_\eps'' + \frac{1}{\eps}W'(q_\eps)\notag \\
  =\, & (1-q)\Big(\eps \frac{4}{\delta^2}\eta'' -
  \frac{4}{\delta\sqrt{2}}\eta'(1+q)
  -\frac{1}{\eps}\eta(1-\eta)(1-q)\Big) \label{eq:mc-q}
\end{align}
is exponentially small in $\eps>0$.
For the distance function we have 
\begin{gather}
  (\Delta d)(\gamma(s)+t\nu(s))\, =\,
  \frac{\kappa(s)}{1+t\kappa(s)} \label{eq:curv-d}
\end{gather}
and for the diffuse mean curvature we obtain
\begin{gather}
  -\eps\Delta u_\eps + \frac{1}{\eps}W'(u_\eps) \,=\, -\eps q_\eps'' +
  \frac{1}{\eps}W'(q_\eps) + \eps q_\eps' \Delta
  d. \label{eq:curv-limsup}
\end{gather}
Therefore
\begin{align}
  & \int_{B_1(0)} \big(-\eps \Delta u_\eps + \frac{1}{\eps}W'(u_\eps)\big) |\nabla
  u_\eps|\\
  =\,& \int_0^L \int_{-\delta}^\delta \Big(-\eps q_\eps''(t) +
  \frac{1}{\eps}W'(q_\eps(t)) + \eps
  q_\eps'(t)\frac{\kappa(s)}{1+t\kappa(s)} \Big) q_\eps'(t)
    (1+t\kappa(s))\,dt\,ds \notag\\
  =\,&   \Big(\int_0^L \kappa(s)\,ds\Big)\int_{-\delta}^\delta \eps
  q_\eps'(t)^2\,dt \notag\\
  & + 2\Big(\int_0^L \kappa(s)\,ds\Big)\int_{\frac{\delta}{2}}^{\delta}
  \Big(-\eps q_\eps''(t) +  \frac{1}{\eps}W'(q_\eps(t))\Big) tq_\eps'(t)\,dt.
\end{align}
Since $\gamma$ is closed and simple we have $\int_0^L
\kappa(s)\,ds=2\pi$. Therefore
\begin{align}
  T_\eps(u_\eps) - 2\pi \,&=\, -\frac{4\pi}{c_0}\int_\delta ^\infty \eps
  q_\eps'(t)^2\,dt + \frac{4\pi}{c_0}\int_{\frac{\delta}{2}}^{\delta}
  \Big(-\eps q_\eps''(t) +  \frac{1}{\eps}W'(q_\eps(t))\Big)
  tq_\eps'(t)\,dt \notag\\
  &=:\, R_\eps^{(T)} \label{eq:approx-T}
\end{align}
and similarly as above one shows that this term is exponentially small in
$\eps>0$. 
\\[2mm]
\underline{Step 4:}
As above we deduce that
\begin{align}
  &\Beps(u_\eps)\notag\\
  =\,& \frac{1}{c_0} \int_0^L \int_{-\delta}^\delta
  \frac{1}{\eps}\Big(-\eps q_\eps''(t) + 
  \frac{1}{\eps}W'(q_\eps(t)) + \eps
  q_\eps'(t)\frac{\kappa(s)}{1+t\kappa(s)} \Big)^2
    (1+t\kappa(s))\,dt\,ds \notag\\
  =\,& \frac{1}{c_0} \int_0^L 
  \int_0^{\delta} {\eps}q_\eps'(t)^2
  \kappa(s)^2\Big(\frac{1}{1+t\kappa(s)}+\frac{1}{1-t\kappa(s)}
  \Big)\,dt\,ds \notag\\
  &+ \int_0^L \int_0^{\delta} \Big(-\eps
  q_\eps''(t)+\frac{1}{\eps}W'(q_\eps(t))\Big)^2
  \frac{1}{\eps}(1+t\kappa(s))\,dt\,ds \notag\\
  =\,& \Big(\int_0^L \kappa(s)^2\,ds\Big)\frac{2}{c_0}\int_0^{\frac{\delta}{2\eps}}
  q'(t)^2\,dt \notag\\
  &+ \eps^2 \frac{2}{c_0}\int_0^L\int_0^{\delta} q'(t)^2
  t^2\kappa(s)^2 \frac{1}{1-\eps^2 t^2\kappa(s)^2}\,dt\,ds \notag\\
  &+\frac{2}{c_0} \int_0^L 
  \int_{\frac{\delta}{2}}^{\delta} \frac{1}{\eps}q_\eps'(t)^2
  \kappa(s)^2\Big(\frac{1}{1+t\kappa(s)}+\frac{1}{1-t\kappa(s)}
  \Big)\,dt\,ds \notag\\
  & + \int_0^L \int_0^{\delta} \Big(-\eps
  q_\eps''(t)+\frac{1}{\eps}W'(q_\eps(t))\Big)^2
  \frac{1}{\eps}(1+t\kappa(s))\,dt\,ds \label{eq:limsup-B1}
\end{align}
The last two terms on the right-hand side are exponentially small and we
finally obtain \eqref{eq:limsup-B} with $R_\eps^{(B)}=O(\eps^2)$. 
\end{proof}
We next can prove the general case.
\begin{proposition}\label{prop:limsup}
Let $\gamma\in \overline{M_L}$. Then the same conclusions as in Lemma \ref{lem:limsup} hold.
\end{proposition}
\begin{proof}
By Proposition \ref{prop:alt-char} we can approximate $\gamma$ strongly in $W^{2,2}([0,L])$ by a sequence of closed simple $C^2$-curves $(\gamma_k)_{k\in \N}$ that satisfy \eqref{eq:appro2} and \eqref{eq:appro3}. In particular $\gamma_k \in M_{L}$ and since $\gamma_k \to\gamma$ strongly  in $W^{2,2}([0,L])$ 
\begin{gather}
  \Bc(\gamma_k)\,\to\, \Bc(\gamma)
\end{gather}
as $k\to\infty$.
Lemma \ref{lem:limsup} yields a sequence of functions $(u_{\eps,k})_{k\in \N}$ that satisfy \eqref{eq:bdry-cdt} and
\begin{gather*}
   \F_\eps(u_{\eps,k})\,\to\, \Bc(\gamma_k)
 \end{gather*}
as $\eps\to 0$. 
Choosing now a suitable diagonal sequences proves the claim.
\end{proof}

\section{Numerical simulations}
\label{sec:numerics}
In  order to demonstrate the feasibility of the above phase field approach to model confined elastic
curves we present some numerical results. To be exact, we use a finite element approach
to discretize a viscous gradient flow of the energy $\F_\eps$, after some modifications described below,
in space and advance the equation in time using a first order fully implicit scheme. For other numerical
approaches to a diffuse interface approximation of constrained Willmore flow see for example~\cite{Lowengrub_09a, DuWang05, DuWang04}.

\subsection{Evolution equation in the numerical simulations}
As it turns out, for finite epsilon, the numerical method does not always yield a perfect transition layer.
For large prescribed length $L$ it can be energetically favorable to not follow the optimal profile
of the transition layer---thus increasing the value of the diffuse length functional---in cases where two transistion
layers were close together. It was therefore necessary to introduce a penalty for the discrepancy of the
phase field to the optimal profile. This term is
\begin{gather*}
   M_\eps(u) = \sigma_\textrm{mis} \int_{B_1(0)} \left( \frac{\eps}{2}\left|\nabla u \right|^2 - \frac{1}{\eps} W(u)\right)^2.
\end{gather*}
It is evident from the proof of Lemma~\ref{lem:limsup} that the addition of such a term does not change the
construction of the recovery sequence and simply vanishes in the limit of small $\eps$ if $\sigma$ 
scales as some power of $\frac{1}{\eps}$.

Unfortunately, the non-differentiability of the factor $\left|\nabla u\right|$ in the diffuse winding number proves to be another problem for the
gradient flow. Its gradient yields $\frac{\nabla u}{\left| \nabla u\right|}$, so the second derivative blows up where $\left|\nabla u\right|$ 
vanishes. Using the fact that the discrepancy of the phase field and the optimal profile have to vanish, we have
\begin{gather*}
\left|\nabla u\right| = \frac{\sqrt{2}}{\eps} \sqrt{W(u)} =  \frac{1}{\sqrt{2}\eps} \left| 1-u^2\right|.
\end{gather*}
The second derivative---which is necessary for the Newton-Raphson iteration used in the implicit time integration---of this term still blows up when $u=\pm 1$, however, the phase field should remain in the
interval $[0,1]$. For the computation, we thus simply leave out the absolute value in this term and observe that the phase
field behaves nicely in the simulation.

In conclusion, we numerically compute the viscous gradient flow of the energy
\begin{gather}
\label{eq:numeq}
\overline{\mathcal{F}_\varepsilon}(u) = \Beps(u) + \eps^{-\alpha} \left( L_\eps(u) - L \right)^2 + 
c_\beta \eps^{-\beta} \left( \overline{T}_\eps(u) - T \right)^2 + M_\eps(u),
\end{gather}
where 
\begin{gather*}
\overline{T}_\eps(u) = \frac{1}{c_0} \int_{B_1(0)} \Big(-\eps\Delta u +
  \frac{1}{\eps}W(u)\Big) \frac{1}{\sqrt{2}\eps}\left(1-u^2\right).
\end{gather*}
The boundary conditions are clamped, i.e., $u=-1$ on the boundary of the domain
and the normal derivative of $u$ on the boundary vanishes.

\subsection{Numerical method}
The space discretization of~\eqref{eq:numeq} requires some care, since its weak
formulation requires $u$ to be in $H^2(B_1(0))$, making it impossible to use
a piecewise linear interpolation directly. While there are several options to
resolve this problem, we resort to using a conforming, i.e., continuously
differentiable finite element discretization. To this end, we construct basis
functions derived
from Loop subdivision surfaces, which can be thought of as a
generalization of multivariate splines to tessellations of arbitrary
topology~\cite{Loop_92a}. The use of subdivision
surfaces for this problem has been suggested in~\cite{Cirak_00a}, 
where one can also find a description of convergence properties.
In addition, we use the method
described in~\cite{Biermann_00a} to fix the clamped boundary
conditions. The computational domain is a disk of radius one, 
discretized using \textsf{distmesh}~\cite{Persson_04a}.
In order to advance the system in time, we use a simple first order implicit
Euler scheme, since accuracy of the time integration is not our
primary concern. 

\subsection{Simulation parameters and results}
\begin{table}
\centering
\begin{tabular}{lllll}
\bf Name & \bf Length & \bf Winding Number & \bf Mismatch &  \bf Length \\
\bf & \bf Constraint& \bf Constraint& \bf Penalty&\bf Target \\ 
Circle 1 & off & off & off & n/a \\
Circle 2 & off & on & on & n/a \\
Relaxation &  on & on &on & $8.7838$ \\
Topology 1 & on & off &on & $8.7838$ \\
Topology 2 & on & on &on & $8.7838$ \\
\end{tabular}
\caption{The parameters for the numerical experiments}
\label{tab:simulations}
\end{table}
We use a triangulation of the domain consisting of 17\,813 faces. The transition
length $\eps$ is kept fixed at 0.025, much larger values produced a significant
mesh effect. The parameters $\alpha$ and $\beta$ are fixed at 2 and $c_\beta = 3$. The mismatch
penalization $\sigma_\textrm{mis}$ is $0.02\varepsilon^{-2}$.

For the numerical method it is essential to impose initial conditions that already are close
to an optimal profile of a simple closed curve. To generate such initial conditions, we take
black and white image to represent the interior and exterior of the initial curve, apply a
Gaussian blur and use the grayscale data as the initial function values. It is then necessary
to relax this initial condition in order for it to be close enough to an optimal profile
for the penalty terms to make sense. We thus, in the beginning, chose a small timestep ($10^{-6}$),
and slowly increase the penalization of the length- and the topological constraint. The initial conditions
are plotted after this relaxation phase which lasts 200 timesteps.
The timestep
is then increased to the regular value of $10^{-5}$. 
In addition, we slowly increase the target value $L$ for the length constraint, starting at
the value of the diffuse length functional at the initial condition (after relaxation). For comparison,
we also provide some simulation results lengh or winding number constraints.

In the following, we briefly describe the simulation results. The parameters for
the various simulations are indicated in Table~\ref{tab:simulations}. There, ``on'' for a penalty term
means that the respective term is used as in equation~\eqref{eq:numeq}. ``Off'' means the term is not
present in the energy used for the computation.
\begin{figure}
\centering
\subfigure[Time vs. radius plot]{
\includegraphics[width=0.45\textwidth]{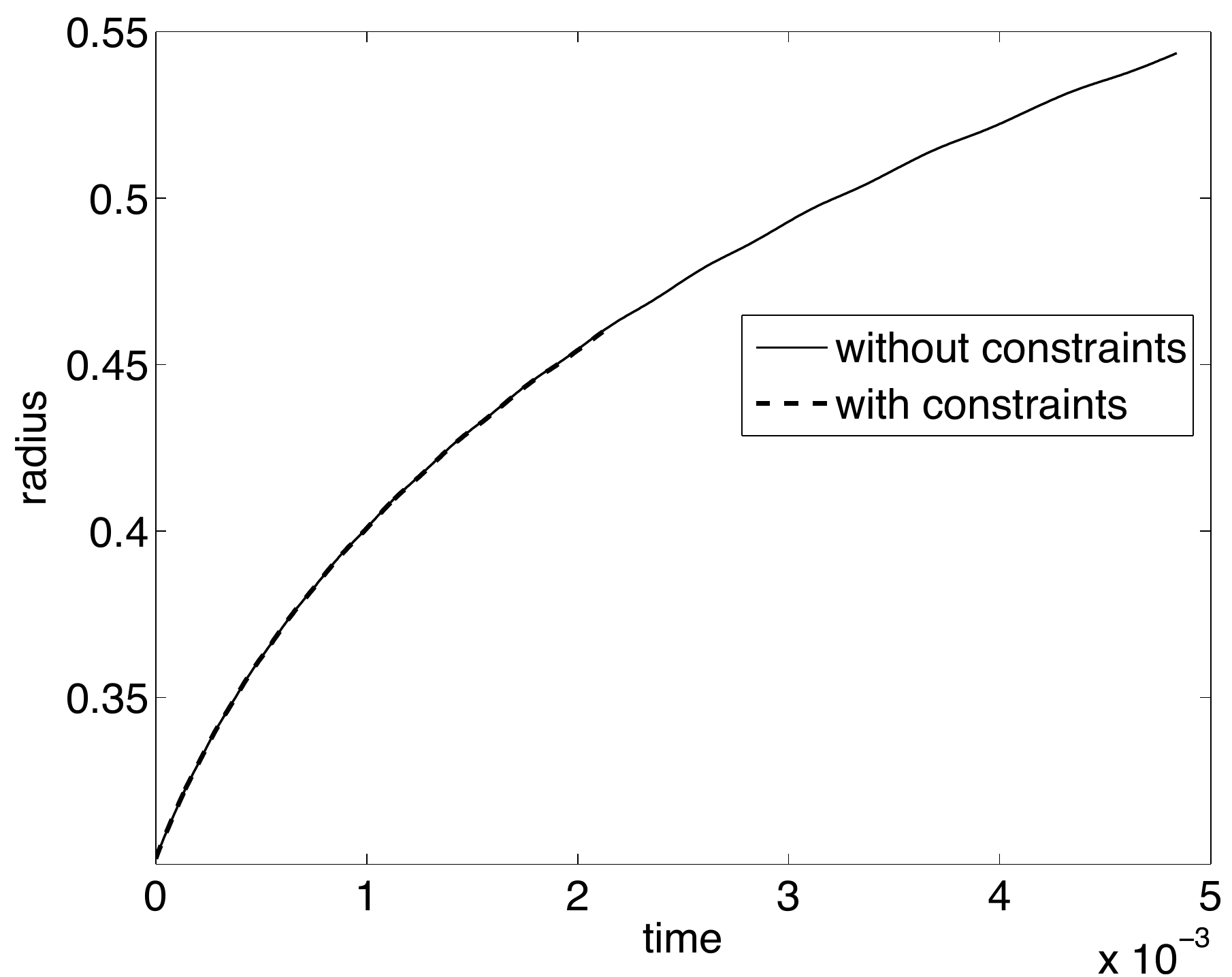}
\label{fig:circ_time_radius}
}
\subfigure[Difference between the expanding circle solution with and without winding number and mismatch penalization]{
\includegraphics[width=0.45\textwidth]{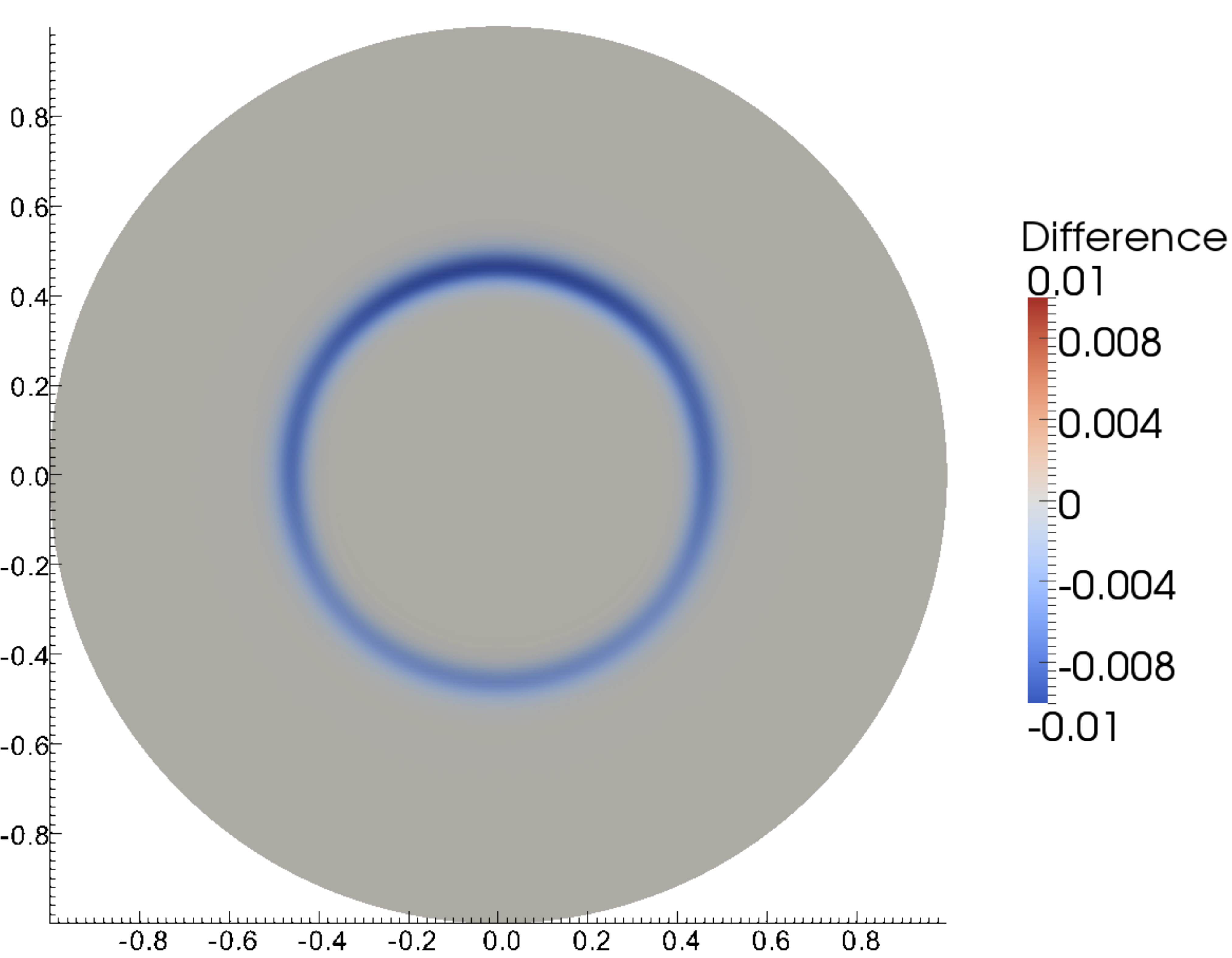}
\label{fig:circ_diff}
}
\caption{Expanding circle}
\label{fig:sim_circle}
\end{figure}

\begin{figure}
\centering
\subfigure[Initial condition (after initial relaxation)]{
\includegraphics[width=0.45\textwidth]{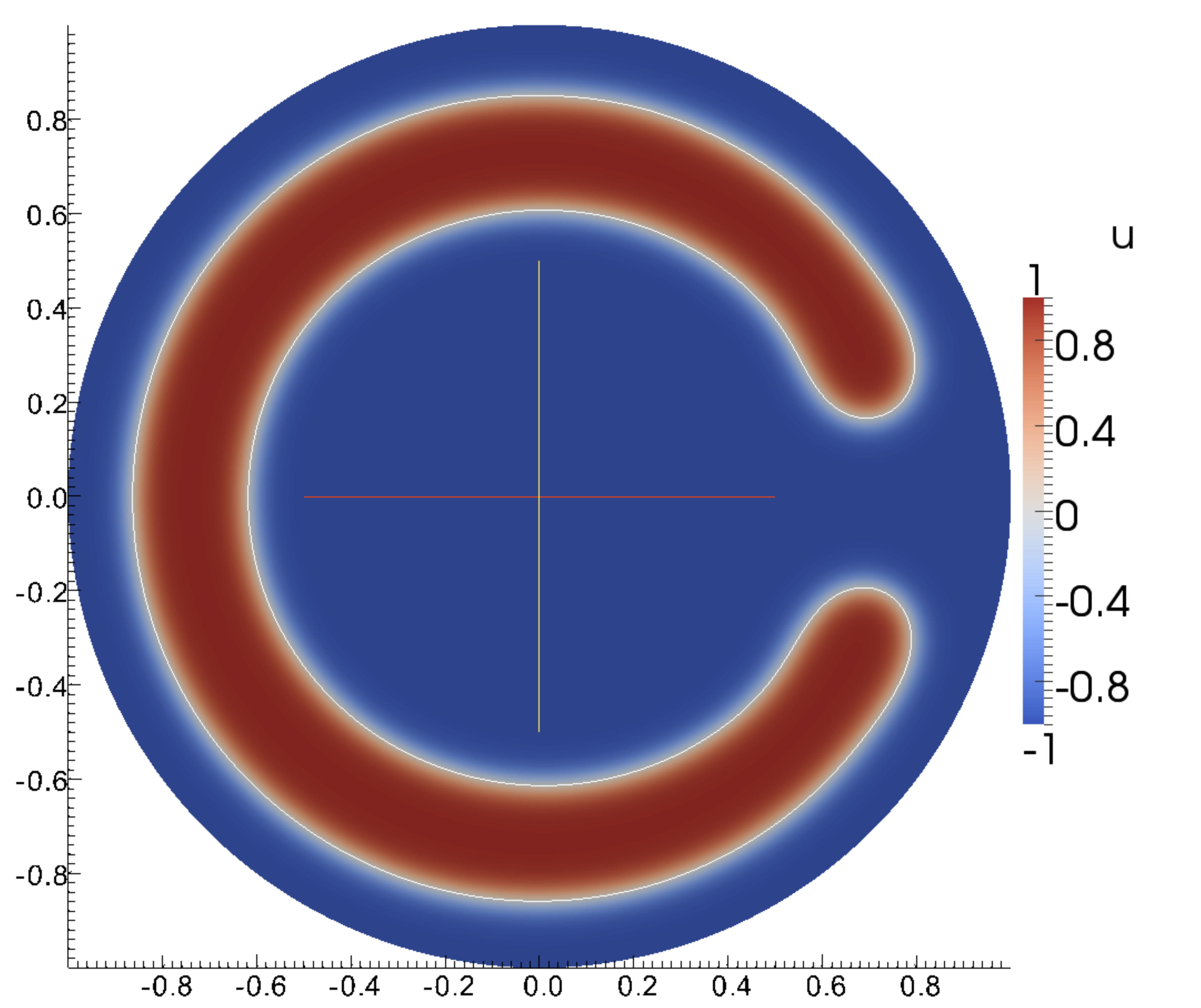}
\label{fig:sym1_initial}
}
\subfigure[Zero level sets of the intermediate stages, the arrow indicates the direction of the flow.]{
\includegraphics[width=0.4\textwidth]{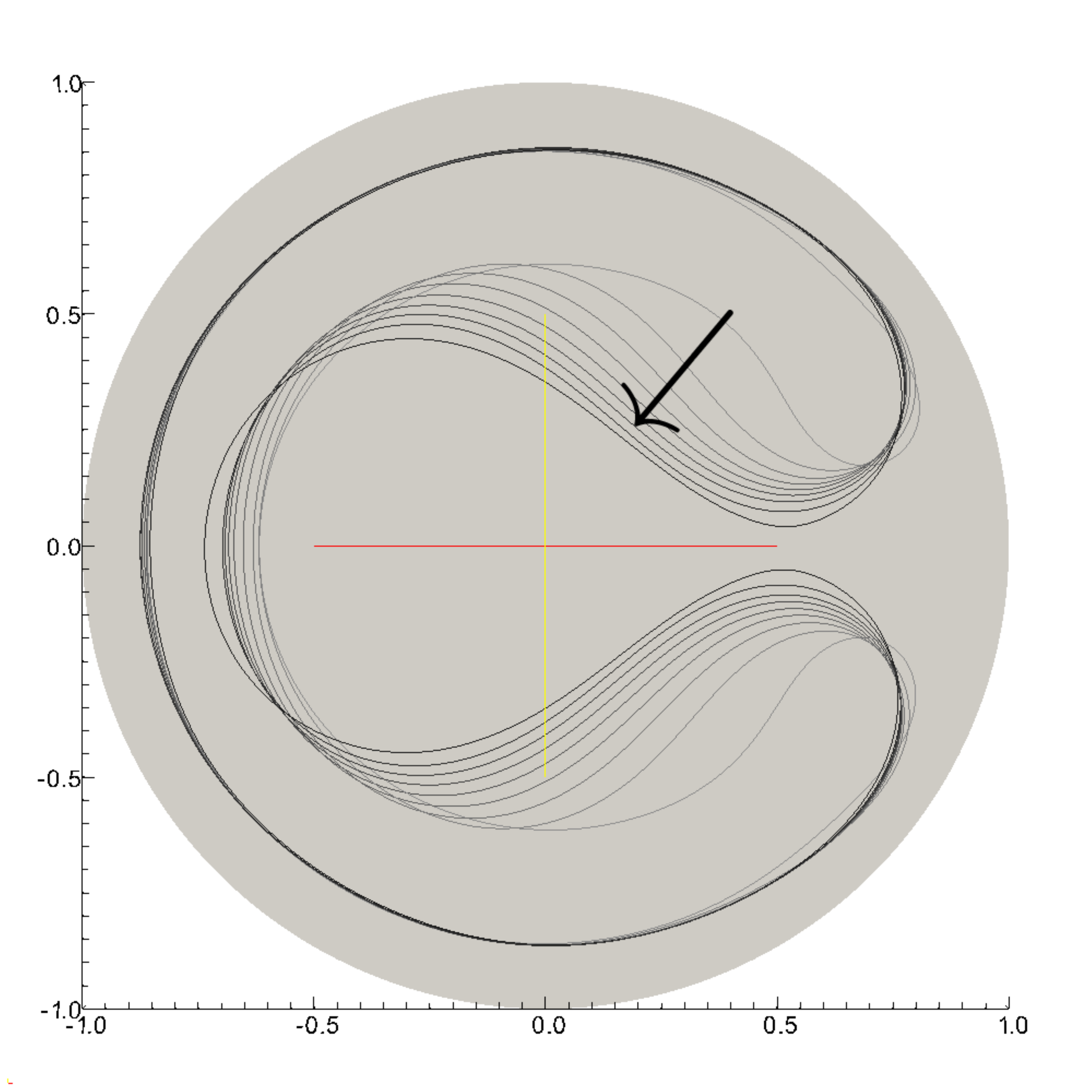}
\label{fig:sym1_intermed}
}
\subfigure[Zero level set of the equilibrium state. Elastic energy density is shown as overlay.]{
\includegraphics[width=0.45\textwidth]{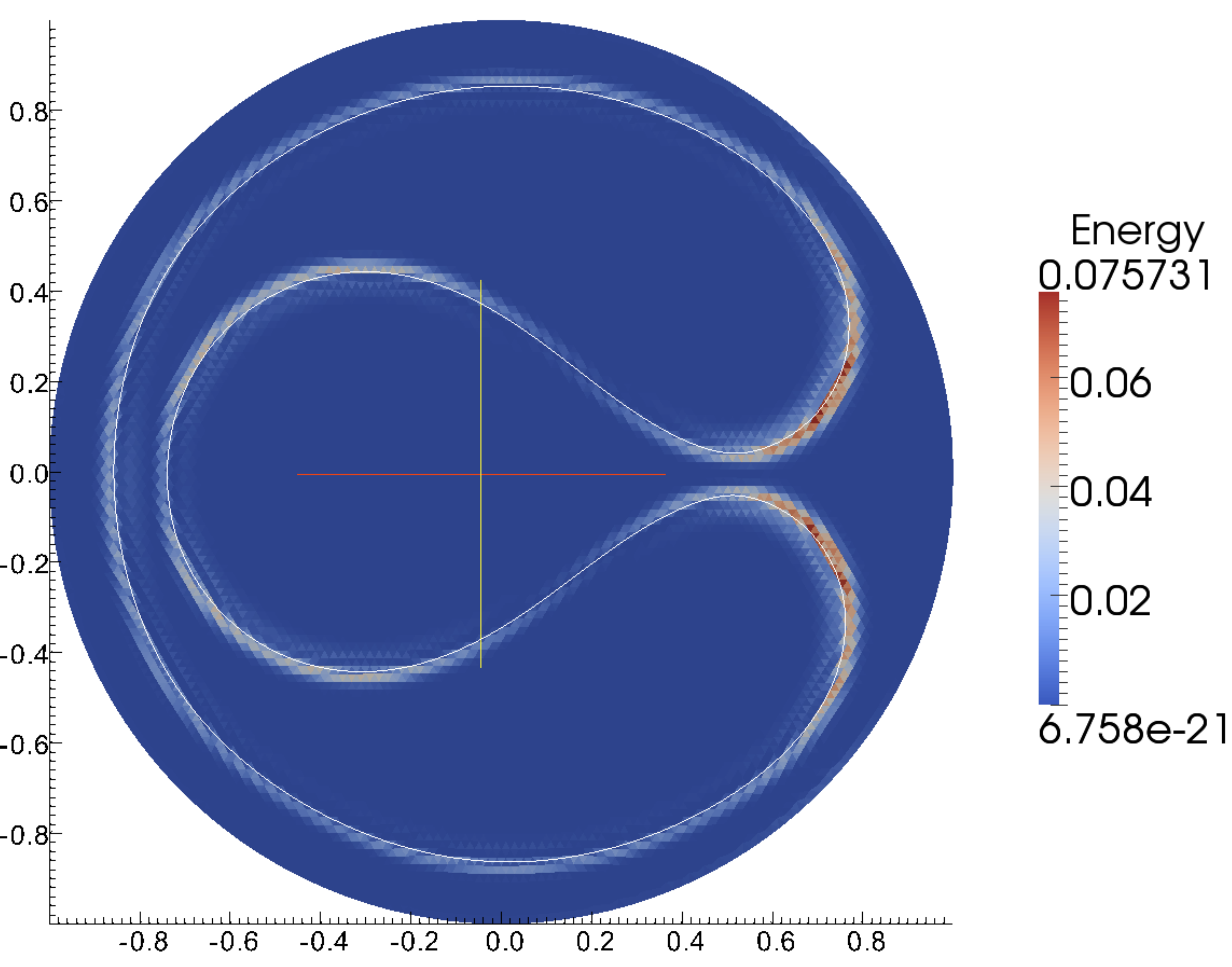}
\label{fig:sym1_final}
}
\caption{Mirror symmetric initial condition}
\label{fig:symmetry1}
\end{figure}

\begin{figure}
\centering
\subfigure[Initial Condition (after initial relaxation)]{
\includegraphics[width=0.45\textwidth]{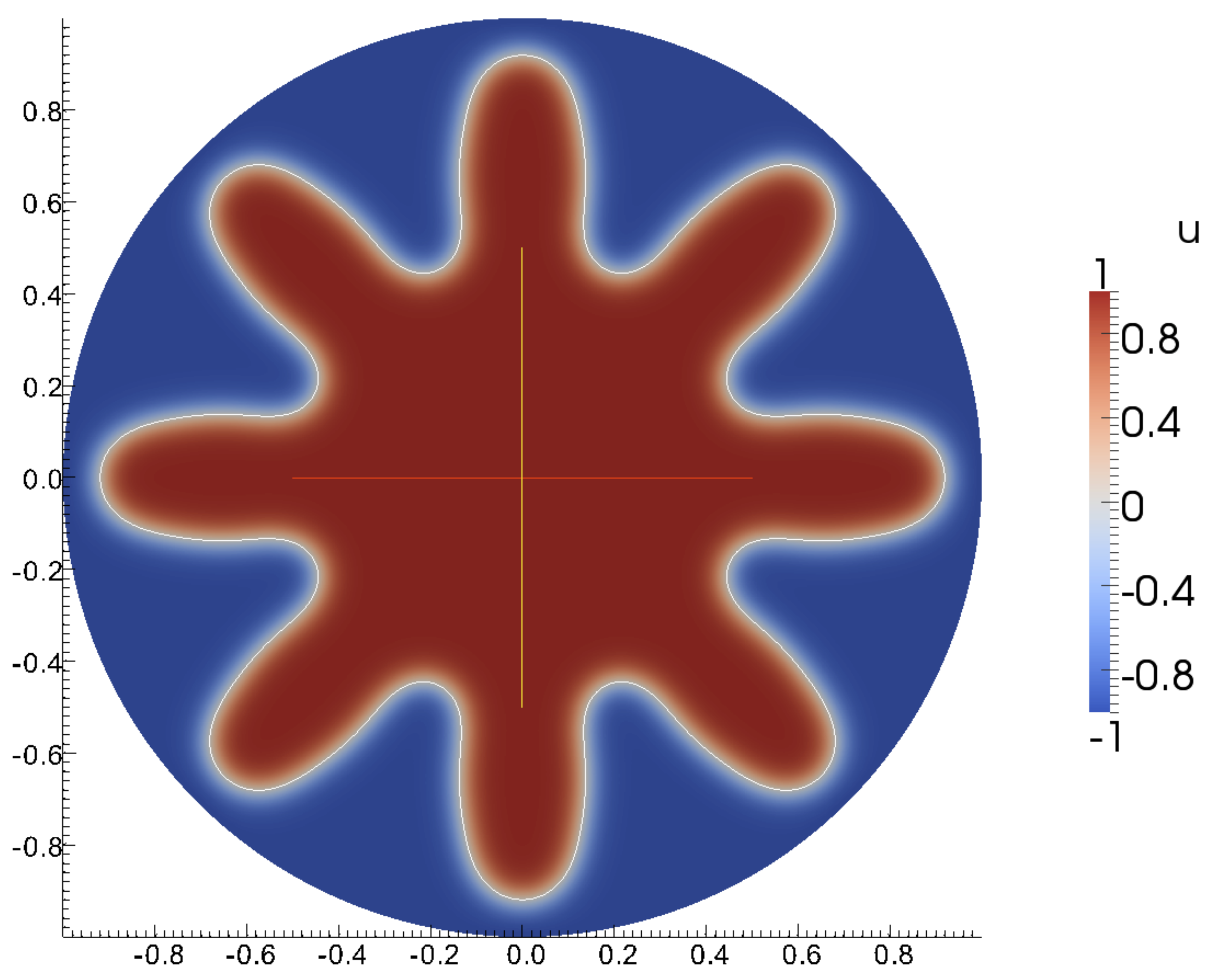}
\label{fig:sym2_initial}
}
\subfigure[Critical stage of topological transition without topological constraint. Diffuse winding number is shown as overlay.]{
\includegraphics[width=0.45\textwidth]{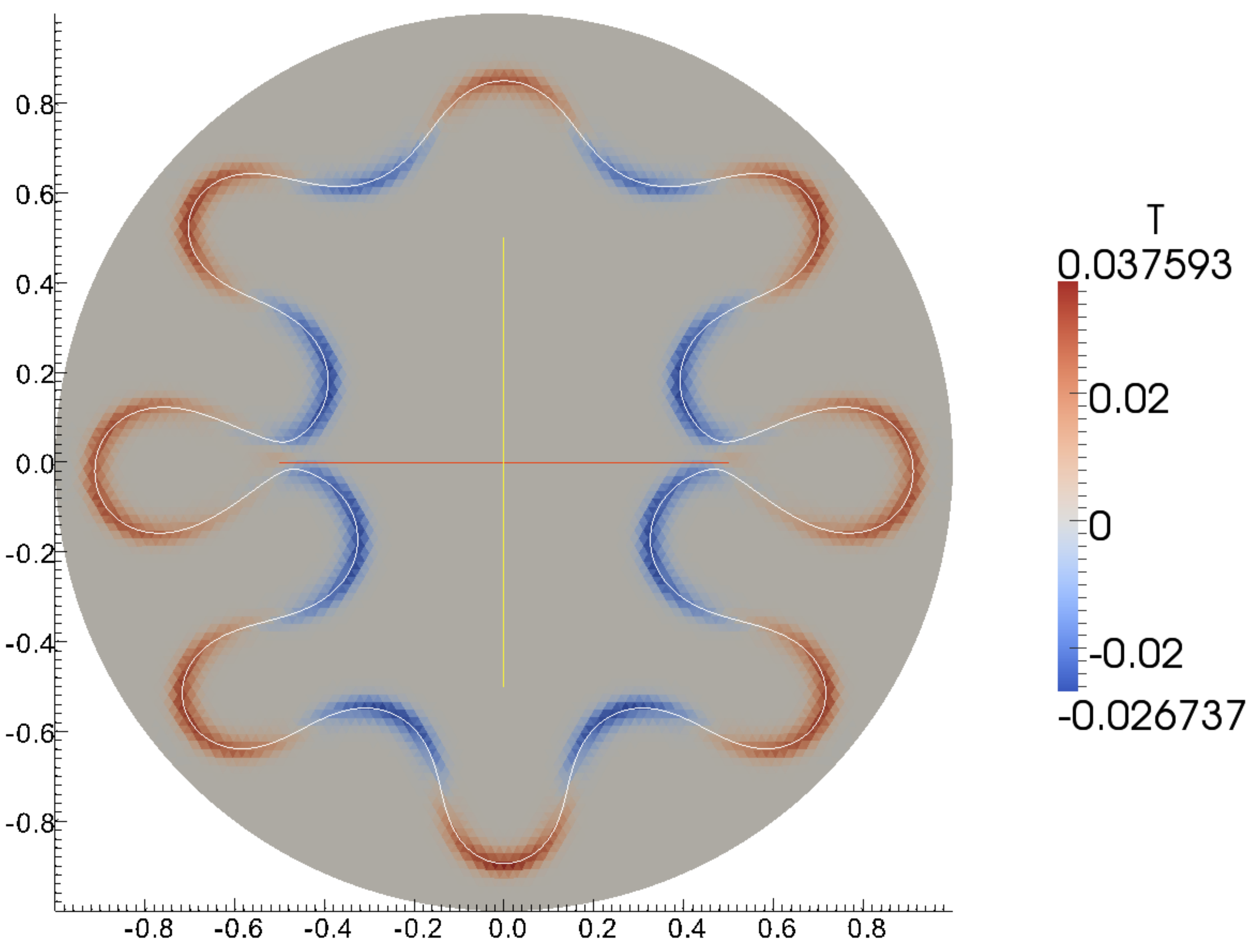}
\label{fig:sym2_noT_crit}
}
\subfigure[Critical stage of topological transition with topological constraint. Diffuse winding number is shown as overlay.]{
\includegraphics[width=0.45\textwidth]{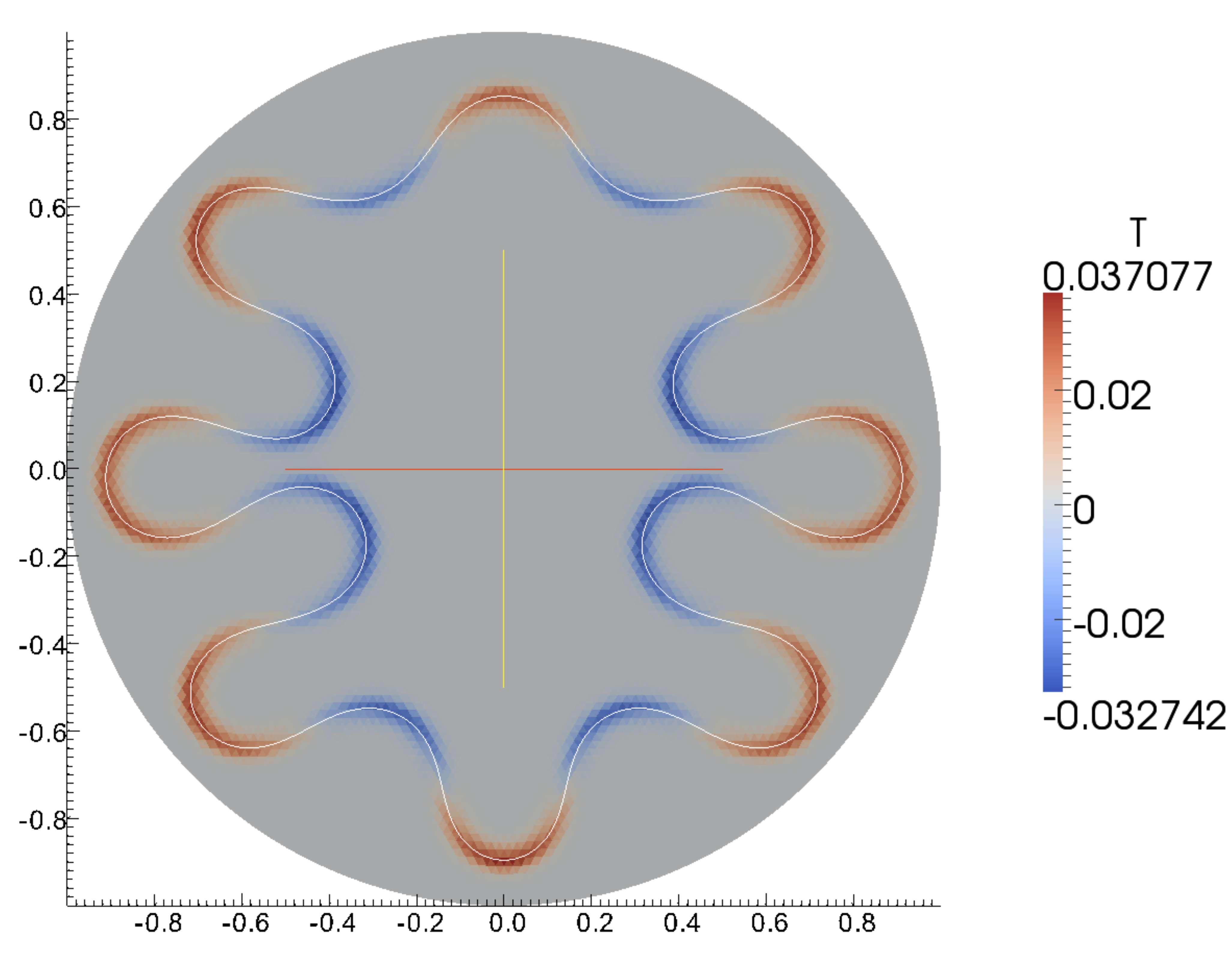}
\label{fig:sym2_crit}
}
\subfigure[Past the critical stage of topological transition without topological constraint. Diffuse winding number is shown as overlay.]{
\includegraphics[width=0.45\textwidth]{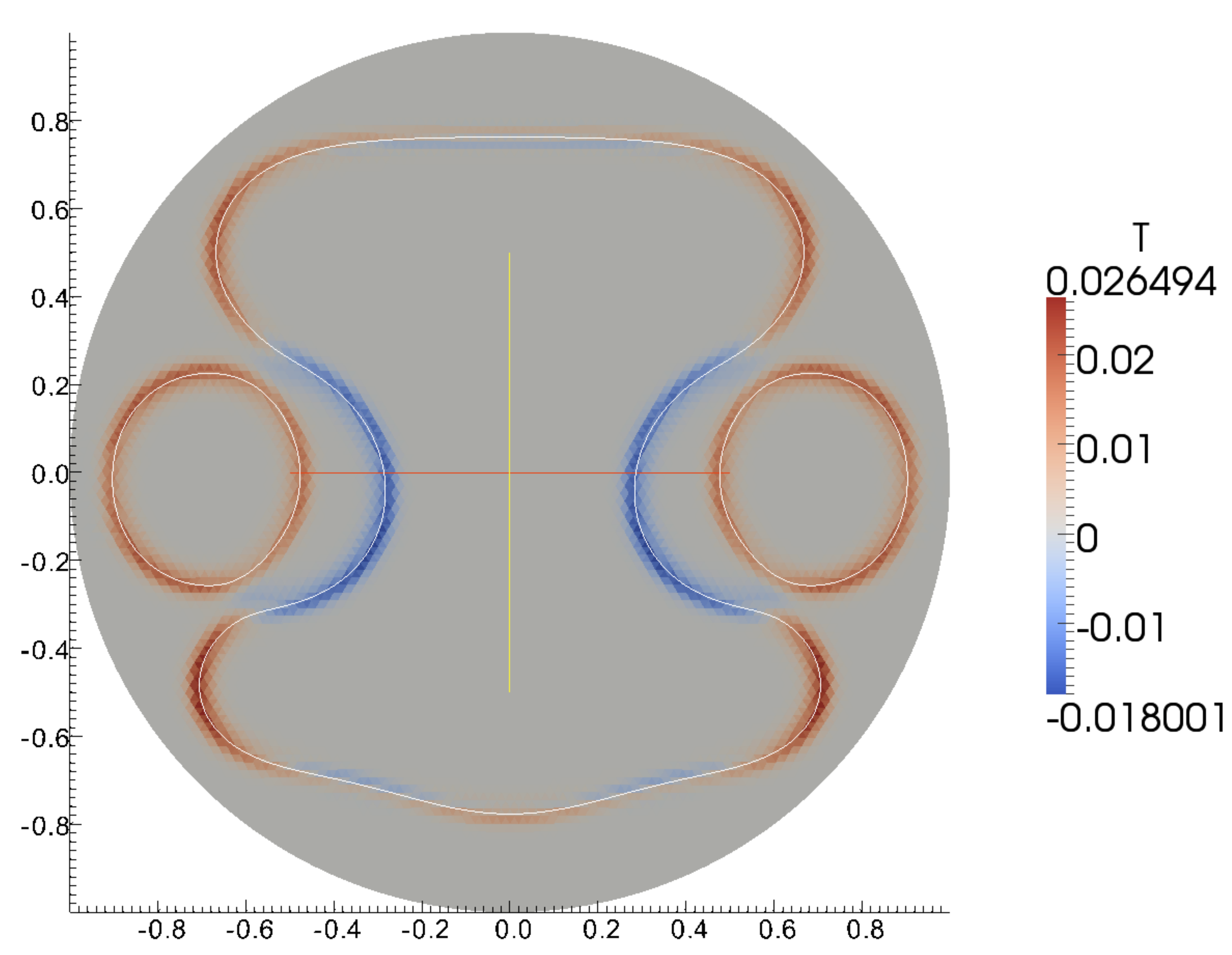}
\label{fig:sym2_noT_past}
}
\subfigure[Past the critical stage of topological transition with topological constraint. Diffuse winding number is shown as overlay.]{
\includegraphics[width=0.45\textwidth]{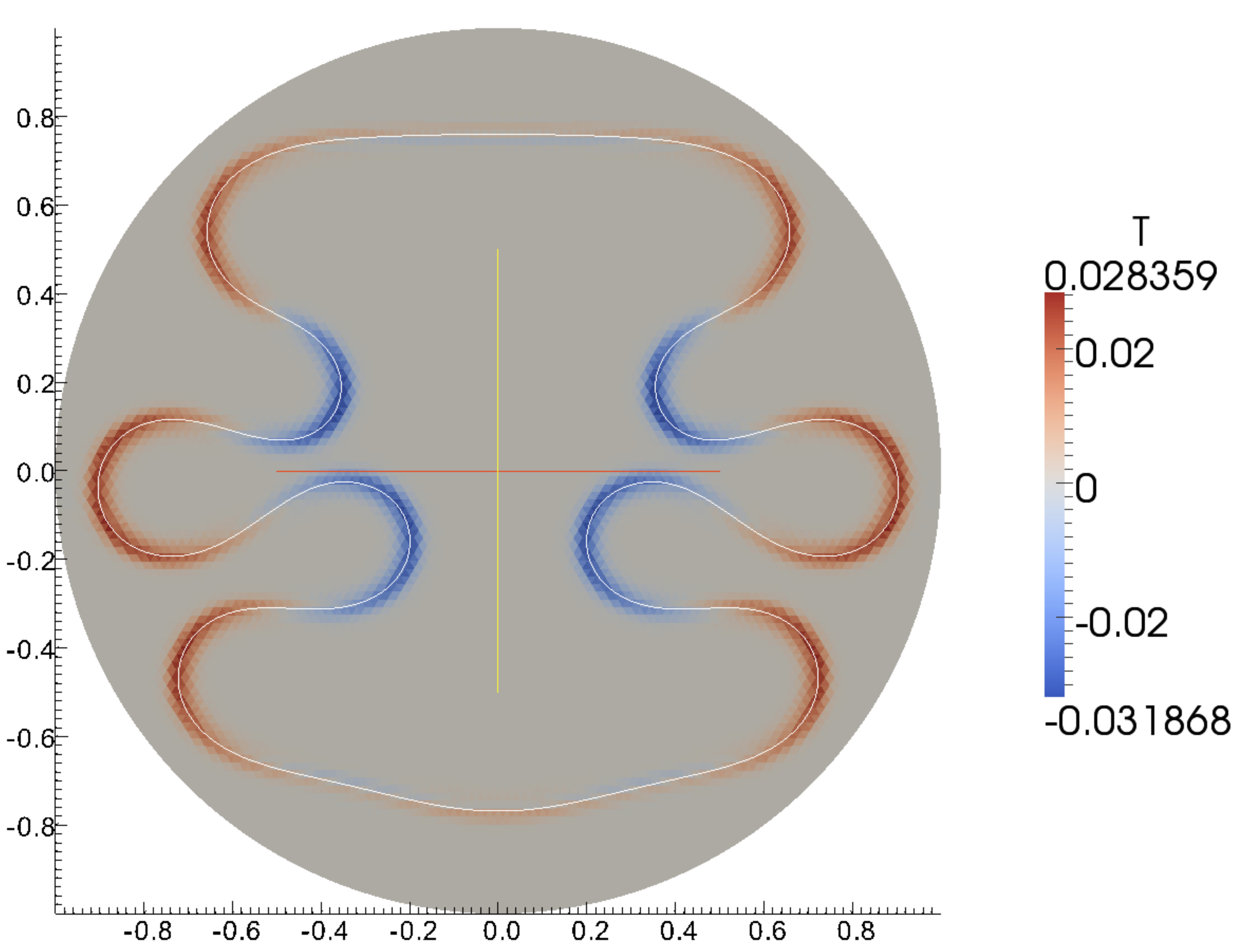}
\label{fig:sym2_past}
}

\caption{Rotationally symmetric initial condition}
\label{fig:symmetry2}
\end{figure}

\paragraph{\bf Expanding circle (Circle 1--2)}
Figure~\ref{fig:sim_circle} shows the expansion of a circle of initial
radius $0.3$ with and without length constraint. While it is clear that the phase field approximation of Euler's elastica energy $\Beps$ alone
gives a good approximation for the Willmore flow of a radially symmetric initial condition, one can see from the radius-vs.-time plot that the
topological constraint does not influence this rate of expansion. The difference of the phase field of the two simulations
(with and without topological and mismatch penalty) is shown in Figure~\ref{fig:circ_diff}. Note that the maximum of
the deviation is small compared to $1$.

\paragraph{\bf Relaxation of a folded mirror-symmetric structure (Relaxation)}
It is clear that the gradient flow routine will only find local minima of the energy, and it stands to
reason that there are many such local minima. We want to investigate the relaxed energy of the gradient flow with initial
condition shown in Figure~\ref{fig:sym1_initial}. Figure~\ref{fig:sym1_intermed} illustrates the evolving surface.
The final relaxed state, with its diffuse energy overlaid, can be seen in Figure~\ref{fig:sym1_final}. The final energy is $33.6$.

\paragraph{\bf Topological transition (Topology 1--2)}
We investigate the effectiveness of the penalization of the diffuse winding number as it differs from $2\pi$. To this end, we
start a simulation with fairly high energy in the state illustrated in Figure~\ref{fig:sym2_initial}. Figures~\ref{fig:sym2_noT_crit} 
and~\ref{fig:sym2_crit} show the state at $t=0.08$ for the simulation \emph{not} penalizing the diffuse winding number and penalizing
the diffuse winding number, respectively. One can clearly see that the simulation without penalization is getting close to pinching off at 
two positions. Finally, one can see that a topological transition occured in Figure~\ref{fig:sym2_noT_past}, while the curve in
Figure~\ref{fig:sym2_past} remained simply connected. Both figures are taken at $t=0.1$. The overlaid diffuse winding number functional 
in those figures clearly shows how the topological transition changes the calculated winding number integral.

\begin{appendix}
\section{A better topological constraint}
\label{sec:impro}
The numerical simulations presented in this paper suggest that  the ``topology'' of the diffuse interface is  preserved along the gradient flow of $\F_\eps$ when the initial condition is well-prepared around an element of $M_L$ and $L$ is not too large with respect to the diameter of the domain. However, in general, neither the functional $\F_\eps$ nor the winding number in the sharp-interface setting enforce the correct topology, as the following example shows.

\noindent Consider 
\begin{gather}
  E\,=\,\Big(B_{1/2}(0)\cup B_{1/5}((0,3/4))\Big)\setminus B_{1/4}(0)
\end{gather}
and consider diffuse approximations $u_\eps$ obtained via the construction presented in Lemma \ref{lem:limsup}. It is then easy to see that $|T_\eps(u_\eps)-2\pi|$ is exponentially small in $\eps$. However, $\partial E$ is not in the admissible class $M_L$ as it cannot be parametrized by a single copy of $S^1$.

The reason why our topological constraint does not work properly in this example is due to the fact that $T_\eps(u)$ represents an approximation of the so called winding number, which depends on the orientation induced by $E$ on the connected components of $\partial E$. In particular connected components of $\partial E$ with opposite orientation (such as $\partial B_{1/4}(0)$ and $\partial B_{1/5}((0,3/4))$ in the example above) compensate each other, and do not contribute to the value of the winding number (and consequently to the value of $T_\eps(u)$). A possibility to avoid such problem (firstly in the sharp-interface) is the following.

Let $(\varphi,\Gamma)$ denote a couple constituted by a finite collection $\Gamma\subset B_1(0)$ of  $W^{2,2}$-regular, simple, closed and disjoint curves,  and a function $\varphi\in C^1(\Gamma,[-1,1])$ such that $\vert\nabla_\Gamma\varphi\vert\equiv 0$ on $\Gamma$, that is $\varphi$ assumes a constant value on each of the connected components of $\Gamma$. More precisely let $N\in\N$, $\gamma_i\in W^{2,2}(S^1,B_1(0))$ ($i=1,\dots,N$) be diffeomorphisms such that $(\gamma_i)\cap(\gamma_j)=\emptyset$ for $i\neq j$, and  let $\Gamma=\big((\gamma_1),\dots,(\gamma_N)\big)$   and $\varphi\equiv c_i\in [-1,1]$ on $(\gamma_i)$.  We then set
\begin{gather}\label{eq:top-constr-sh-inter}
A(\varphi,\Gamma):=\int_\Gamma \varphi\, \kappa_\Gamma\,d\Ha^1= 
\sum_{i=1}^{N}c_i\int_{(\gamma_i)}\kappa_{\gamma_i}\,d\Ha^1.
\end{gather}
Being $\gamma_1,\dots\gamma_N$ simple, regular, closed and disjoint curves, we can find $(l_1,\dots,l_N)\in \{1,2\}^N$ such that, setting $\varphi[\Gamma]:\equiv(-1)^{l_i}$ on $(\gamma_i)$, we have
\begin{gather*}
\tilde T(\Gamma):=\inf\{A(\varphi,\Gamma):~\varphi\in BV(\Gamma),\, \vert \nabla_\Gamma\varphi\vert(\Omega)=0\}=
A\big(\varphi[\Gamma],\Gamma\big)
\\
=\sum_{i=1}^{N}(-1)^{l_i}\int_{(\gamma_i)}\kappa_{\gamma_i}\,d\Ha^1=
-2\pi N.
\end{gather*}
Hence the functional $\tilde T(\Gamma)$ counts the number of connected components of $\Gamma$, without taking into account of  their orientation. It is then rather natural to look for a phase-fields approximation for $\tilde T(\Gamma)$ in order to get a  constraint on the topology of the diffuse interfaces stronger than the one obtained via $T_\eps(u)$. For this purpose we proceed as follows.
We firstly consider a sequence $A_\eps(\varphi,u)$  of functionals defined on couples $(\varphi,u)\in C^1(\Omega,[-1,1])\times C^2(\Omega)$,  and representing a diffuse interface approximation of $A(\varphi,\Gamma)$, Then, in analogy with $\tilde T(\Gamma)$, we define a  functional $\tilde T_\eps(u)$ via minimization with respect to $\varphi$ of $A_\eps(\varphi,u)$. Finally we define the new topological constraint penalizing deviations of $\tilde T_\eps(u)$ from $2\pi$. More precisely we start setting
\begin{gather}
A_{\eps}(\varphi,u)\,:=\, \eps^{\gamma}\int_{B_1(0)} |\nabla\varphi|\,dx +
  \frac{1}{\eps^\gamma} \int_{B_1(0)} \big| \nabla u^\perp \cdot
  \nabla\varphi\big| \eps |\nabla u| \,dx\notag\\
  - \frac{1}{c_0}
  \int_{B_1(0)}\Big(-\eps\Delta u + \frac{1}{\eps}W'(u)\Big)\varphi
  |\nabla u| \,dx.
  \label{eq:def-Aeps}
 \end{gather} 
$A_\eps(\varphi,\Gamma)$ formally presents a diffuse interfaces approximation of $A(\varphi,\Gamma)$. The second term in 
\eqref{eq:def-Aeps} represents a  penalization of order $\eps^{-\gamma}$ of the integral, with respect to the diffuse-length measure $\eps\vert\nabla u\vert^2\mathcal L_{\lfloor B_1(0)}$, of the variations  of $\varphi$ along the diffuse interface $\{y\in B_1(0):~\nabla u(y)\neq 0\}$. Hence this term corresponds to a relaxation (at the diffuse interface  level) of the (sharp-interface) constraint $\vert\nabla_\Gamma \varphi\vert\equiv 0$ on $\Gamma$. The third term, as  we have  already seen in \ref{eq:approx-T}, can be thought of as a phase-fields approximation of  $\int_\Gamma \varphi\,\kappa_\Gamma\,d\Ha^1$. Finally the first term, whose contribution is infinitesimal of order $\eps^\gamma$, is needed in order to ensure compactness in $BV(\Omega,[-1,1])$ when minimizing $A_\eps(\varphi,u)$ with respect to the first variable.
In fact, we remark that, fixed  $u\in C^2(\Omega)$,  we can apply the direct method of calculus of variations, and obtain
the existence of a function $\varphi[u]\in BV(\Omega,[-1,1])$ such that
\begin{gather*}
\overline{A_\eps}(\varphi[u],u)=\inf_{\varphi\in C^1(\Omega,[-1,1]) }  A_{\eps}(\varphi,u),
\end{gather*}
where $\overline{A_\eps}(\cdot,u)$ denotes the lower semi-continuous envelope of $A_\eps(\cdot,u)$ with respect to the weak convergence in $BV(\Omega,[-1,1])$. Hence we define
\begin{gather}
  \tilde{T}_\eps(u) \,=\, \frac{1}{c_0} \int_{B_1(0)} \Big(-\eps\Delta u +
  \frac{1}{\eps}W(u)\Big)\varphi[u]|\nabla u|\,dx, \label{eq:def-Teps-tilde}\\
  	\tilde{\F}_\eps(u)\,=\, \Beps(u) + \eps^{-\alpha}\Big(L_\eps(u)-L\Big)^2 +
 \eps^{-\beta}\Big(\tilde{T}_\eps(u)-2\pi\Big)^2 \label{eq:def-F-tilde}
\end{gather}
and remark that when $\varphi[u]\equiv 1$ the functional $\tilde{T}_\eps(u)$ coincides with
the diffuse winding number ${T}_\eps(u)$.

 In order to justify the choice of $\tilde T_\eps(u)$ we first show (see Lemma \ref{lem:Ttilde}) that if $(u_\eps)_{\eps>0}$ is as in Lemma \ref{lem:limsup} the value of $\tilde{T}_\eps(u_\eps)$ still converges to $2\pi$ as $\eps\to 0$.
Eventually, in Proposition \ref{prop:pappa}, we analyze the behavior of  $\tilde T_\eps(\cdot)$ along sequences $\{u_\eps\}_\eps$ approximating (in an ``optimal way'') a finite collection $\Gamma$ of simple closed, disjoint curves in $B_1(0)$, and obtain that  $\tilde T_\eps(u_\eps)$ converges to $\tilde T(\Gamma)$.
\begin{lemma}\label{lem:Ttilde}
Let $u_\eps$ be as in \eqref{eq:def-u-limsup}. Then we have
\begin{gather}
  \lim_{\eps\to 0} \tilde{T}_\eps(u_\eps)\,=\, 2\pi. \label{eq:Ttil}
\end{gather}
\end{lemma}
\begin{proof}
As in the proof of Lemma \ref{lem:limsup} we calculate that up to
exponentially small term
\begin{align}
  & \int_{B_1(0)} \big(-\eps \Delta u_\eps +
  \frac{1}{\eps}W'(u_\eps)\big) \varphi |\nabla  u_\eps|\\
  \approx\,& \int_0^L \int_{-\frac{\delta}{2}}^{\frac{\delta}{2}}
  \eps q_\eps'(t)^2
  \kappa(s)\varphi(\gamma(s)+t\nu(s))\,dt\,ds. \label{eq:lem-Ttil1} 
\end{align}
Writing $\gamma'(s)=(\cos\alpha(s),\sin\alpha(s))^T$ and observing that
$\kappa(s)=-\alpha'(s)$ 
we deduce that
\begin{align}
  \int_0^L \kappa(s)\varphi(\gamma(s)+t\nu(s))\,ds\,=\,& \int_0^L
  \alpha(s)\nabla\varphi(\gamma(s)+t\nu(s))\cdot
  (1+t\kappa(s))\gamma'(s)\,ds \notag\\
  &-\varphi(\gamma(0)+t\nu(0))\big(\alpha(L)-\alpha(0)\big). \label{eq:lem-Ttil2} 
\end{align}
We next observe that
\begin{gather*}
  \gamma'(s)\,=\, \frac{\nabla u^\perp}{|\nabla u|}(\gamma(s)+t\nu(s))
\end{gather*}
and define $\alpha(x)$ as the angle in $[0,2\pi)$ such that
\begin{gather*}
  \frac{\nabla u^\perp}{|\nabla u|}(x) \,=\, \begin{pmatrix}
    \cos\alpha(x)\\\sin\alpha(x)\end{pmatrix}.
\end{gather*}
We then obtain for the first term on the right-hand side of
\eqref{eq:lem-Ttil2}
\begin{align}
  &\Big| \int_0^L \int_{-\frac{\delta}{2}}^{\frac{\delta}{2}}
  \eps q_\eps'(t)^2  \alpha(s)\nabla\varphi(\gamma(s)+t\nu(s))\cdot
  (1+t\kappa(s))\gamma'(s)\,dt \,ds \Big| \notag\\
  =\,& \Big| \int_{\{|d|<\delta/2\}} \alpha(x)\frac{\nabla u^\perp}{|\nabla
    u|}(x)\cdot \nabla\varphi(x) \eps |\nabla u|^2\,dx |\Big| \notag\\
  \leq\,& 2\pi \int_{B_1(0)} \big| \nabla u^\perp\cdot
  \nabla\varphi\big| \eps |\nabla u|. \label{eq:lem-Ttil3}
\end{align}
For the second term on the right-hand side of \eqref{eq:lem-Ttil2} we
have
\begin{align}
  \Big| \int_{-\frac{\delta}{2}}^{\frac{\delta}{2}} \eps q_\eps'(t)^2
  \varphi(\gamma(0)+t\nu(0))\big(\alpha(L)-\alpha(0)\big) \Big| \,&\leq\,
  2\pi c_0. \label{eq:lem-Ttil4}
\end{align}
This shows that for arbitrary $\varphi \in BV(B_1(0);[-1,1])$, up to
exponentially small terms in $\eps$
\begin{gather*}
  A_\eps(\varphi)\,\geq\,  \frac{1}{\eps^\gamma} \int_{B_1(0)} \big|
  \nabla u^\perp \cdot   \nabla\varphi\big| \eps |\nabla u| - 2\pi -
  2\pi \int_{B_1(0)} \big| \nabla u^\perp\cdot 
  \nabla\varphi\big| \eps |\nabla u| \,\geq\, -2\pi.
\end{gather*} 
On the other hand for $\varphi\equiv 1$ we have
\begin{gather*}
  A_\eps(\varphi)\,=\, -T_\eps(u)\,\approx\, -2\pi.
\end{gather*}
This shows that $\lim_{\eps\to 0} \tilde{T}_\eps(u)\,=\, 2\pi$, as
claimed.
\end{proof}
An application of the previous lemma shows that the improved topological constraint in the case of a finite collection of simple, disjoint curves adds up the winding numbers of each single curve. In particular, these configurations are strongly penalized by the modified functional $\tilde{\F}_\eps$.
\begin{proposition}\label{prop:pappa}
Let $E\subset\subset B_1(0)$ be an open subset with $C^2$-bounday $\partial E=\cup_{j=1}^N(\gamma_j)$ where $\gamma_j$ are $C^2$-diffeomorphisms of the unit circle. Let $\{u_\eps\}_\eps\subset C^2(\Omega)$ be  constructed as in Lemma \ref{lem:limsup}  such that in particular $u_\eps\to 2\chi_E-1$ in $L^1(B_1(0))$. Then  
\begin{gather*}
\lim_{\eps\to 0}  \tilde{T}_\eps(u_\eps) \,=\, N 2\pi.
\end{gather*}
\end{proposition}
\begin{proof}
Since jumps of $\varphi$  from $1$ to $-1$  in the space between two connected components of $\partial E$ are infitesimal of order $\eps^\gamma$, it is enough to repeat the proof of Lemma \ref{lem:Ttilde} for each connected component of $\partial E$.   
\end{proof}

\end{appendix}
\bibliography{cec}
\bibliographystyle{abbrv}
\end{document}